\documentclass[12pt,twoside]{amsart}

\usepackage{hyperref}
\usepackage{amsthm, amsmath, amscd, amssymb,centernot}
\usepackage{bm}
\usepackage[all]{xy}
\usepackage[T1]{fontenc}
\usepackage[left=2.5cm,top=2.5cm,bottom=3cm,right=2.5cm]{geometry}
\usepackage{graphicx}
\usepackage{latexsym}
\usepackage{color}
\usepackage{epsfig}
\usepackage{amsopn}
\usepackage{mathrsfs}
\usepackage{array}
\usepackage{tikz}
\usetikzlibrary{er,positioning}

\newcommand{\proj}{\mathbb{P}}
\renewcommand{\P}{\mathbb{P}}
\newcommand{\Z}{\mathbb{Z}}
\newcommand\call{\mathcal{L}}

\newcommand{\eps}{\varepsilon}
\renewcommand{\phi}{\varphi}
\newcommand\wtilde[1]{\widetilde{#1}}
\theoremstyle{plain}
\numberwithin{equation}{section}
\newtheorem{theorem}{Theorem}[section]
\newtheorem*{theorem*}{Theorem}

\newtheorem{lemma}[theorem]{Lemma}
\newtheorem{corollary}[theorem]{Corollary}
\newtheorem{conjecture}[theorem]{Conjecture}
\newtheorem{definition}[theorem]{Definition}
\theoremstyle{definition}
\newtheorem{question}[theorem]{Question}
\newtheorem{problem}[theorem]{Problem}
\newtheorem{remark}[theorem]{Remark}
\newtheorem{example}[theorem]{Example}

\newcommand\newop[2]{\def#1{\mathop{\rm #2}\nolimits}}
\newop\Bl{Bl}
\newop\mult{mult}
\newop\Pic{Pic}

\begin{document}
\title{Negative curves on special rational surfaces}

\thanks{}

\date{\today}

\author[Dumnicki et al.]{Marcin Dumnicki}
\address{Jagiellonian University, Faculty of Mathematics and Computer Scien\-ce, {\L}ojasiewicza~6, PL-30-348 Krak\'ow, Poland}
\email{Marcin.Dumnicki@uj.edu.pl}

\author[]{{\L}ucja Farnik}
\address{Department of Mathematics, Pedagogical University of Cracow,
   Podchor\c a\.zych 2,
   PL-30-084 Krak\'ow, Poland}
\email{Lucja.Farnik@gmail.com}

\author[]{Krishna Hanumanthu}
\address{Chennai Mathematical Institute, H1 SIPCOT IT Park, Siruseri, Kelambakkam 603103, India}
\email{krishna@cmi.ac.in}

\author[]{Grzegorz Malara}
\address{Department of Mathematics, Pedagogical University of Cracow,
   Podchor\c a\.zych 2,
   PL-30-084 Krak\'ow, Poland}
\email{grzegorzmalara@gmail.com}

\author[]{Tomasz Szemberg}
\address{Department of Mathematics, Pedagogical University of Cracow,
   Podchor\c a\.zych 2,
   PL-30-084 Krak\'ow, Poland}
\email{tomasz.szemberg@gmail.com}

\author[]{Justyna Szpond}
\address{Department of Mathematics, Pedagogical University of Cracow,
   Podchor\c a\.zych 2,
   PL-30-084 Krak\'ow, Poland}
\email{szpond@up.krakow.pl}

\author[]{Halszka Tutaj-Gasi\'{n}ska}
\address{Jagiellonian University, Faculty of Mathematics and Computer Scien\-ce, {\L}ojasiewicza~6, PL-30-348 Krak\'ow, Poland}
\email{halszka.tutaj-gasinska@im.uj.edu.pl}

\thanks{
        KH was partially supported by a grant from Infosys
        Foundation and by DST SERB MATRICS grant MTR/2017/000243.
        TS and JS were partially supported by Polish National
        Science Centre grant 2018/30/M/ST1/00148.}

\begin{abstract}
We study negative curves on surfaces obtained by blowing up
special configurations of points in $\proj^2$. Our main results concern the following configurations: very general points on an irreducible cubic, $3$--torsion points on an elliptic curve and nine Fermat points.
As a consequence of our analysis, we also show that the Bounded
Negativity Conjecture holds for the surfaces we consider.
   The note contains also some problems for future attention.
\end{abstract}

\maketitle

\section{Introduction}
   Negative curves on algebraic surfaces are an object of classical interest.
   One of the most prominent achievements of the Italian School of
   algebraic geometry was Castelnuovo's
   Contractibility Criterion.
\begin{definition}[Negative curve]
   We say that a reduced and irreducible curve $C$ on a smooth projective surface is \emph{negative},
   if its self-intersection number $C^2$ is less than zero.
\end{definition}
\begin{example}[Exceptional divisor, $(-1)$-curves]
   Let $X$ be a smooth projective surface and let $P\in X$ be a closed point.
   Let $f:\Bl_PX\to X$ be the blow up of $X$ at the point $P$. Then the exceptional
   divisor $E$ of $f$ (i.e., the set of points in $\Bl_PX$ mapped by $f$ to $P$) is a negative
   curve. More precisely, $E$ is rational and $E^2=-1$. By a slight abuse of language
   we will call such curves simply $(-1)$--curves.
\end{example}
   Castelnuovo's result asserts that the converse is also true, see \cite[Theorem V.5.7]{Hartshorne} or
   \cite[Theorem III.4.1]{BPV}.
\begin{theorem}[Castelnuovo's Contractibility Criterion]
   Let $Y$ be a smooth projective surface defined over an algebraically closed field.
   If $C$ is a rational curve with $C^2=-1$, then there exists a
   smooth projective surface $X$ and a projective morphism
   $f:Y\to X$ contracting $C$ to a smooth point on $X$. In other words, $Y$ is isomorphic
   to $\Bl_PX$ for some point $P\in X$.
\end{theorem}
   The above result plays a pivotal role in the Enriques-Kodaira classification of surfaces.

   Of course, there are other situations in which negative curves on algebraic surfaces appear.
\begin{example}\label{ex: C x C}
   Let $C$ be a smooth curve of genus $g(C)\geq 2$. Then the diagonal $\Delta\subset C\times C$
   is a negative curve as its self-intersection is $\Delta^2=2-2g$.
\end{example}
   It is quite curious that it is in general not known if for a general curve $C$, there are other
   negative curves on the surface $C\times C$, see \cite{Kou93}. It is in fact even more interesting,
   that there is a direct relation between this problem and the famous Nagata Conjecture. This was
   observed by Ciliberto and Kouvidakis \cite{CilKou99}.

There is also a connection between negative curves and the Nagata
Conjecture on general blow ups of $\P^2$.
We recall the following conjecture about $(-1)$-curves which in fact
implies the Nagata Conjecture; see \cite[Lemma 2.4]{CHMR13}.
\begin{conjecture}[Weak SHGH Conjecture] \label{(-1)-curves}
   Let $f: X\to \P^2$ be the blow up of the projective plane $\P^2$ in
   general points $P_1,\ldots,P_s$. If $s\geq 10$, then the only negative curves
   on $X$ are the $(-1)$-curves.
\end{conjecture}
   On the other hand, it is well known that already a blow up of $\P^2$ in $9$ general points
   carries \emph{infinitely} many $(-1)$--curves.

   One of the central and widely open problems concerning negative curves on algebraic surfaces
   asks whether on a fixed surface negativity is bounded. More precisely, we have
   the following conjecture (BNC in short).  See \cite{BNC} for an extended introduction to this problem.
\begin{conjecture}[Bounded Negativity Conjecture]\label{bnc}
   Let $X$ be a smooth projective surface. Then there exists a number $\tau$ such that
   $$C^2\geq \tau$$
   for any reduced and irreducible curve $C\subset X$.
\end{conjecture}
   If the Conjecture holds on a surface $X$, then we denote by $b(X)$ the largest
   number $\tau$ such that the Conjecture holds. It is known (see
   \cite[Proposition 5.1]{BNC}) that if
   the negativity of reduced and irreducible curves is bounded below,
   then the negativity of all reduced curves is also bounded below.


Conjecture \ref{bnc} is known to fail in the positive characteristic;
see \cite{Har10,BNC}.
 In fact Example \ref{ex: C x C} combined with the action of the Frobenius morphism provides
   a counterexample. In characteristic zero, Conjecture \ref{bnc} is
   open in general. It is easy to prove BNC in some cases; see Remark
\ref{anti-canonical} for an easy argument when  the anti-canonical
divisor of $X$ is nef. However, in many other cases the
conjecture is open. In particular the following question is open and
answering it may lead to a better understanding of Conjecture \ref{bnc}.

\begin{question}\label{que: birational}
Let $X,Y$ be smooth projective surfaces and suppose that $X$ and $Y$
are birational and Conjecture \ref{bnc} holds for $X$. Then does
Conjecture \ref{bnc} hold for $Y$ also?
\end{question}

This is not known even in the simplest case, when one of surfaces is $\P^2$
(where Conjecture \ref{bnc} obviously holds) and the other is a blow up of $\P^2$.
If we blow up general points, then this is governed by Conjecture \ref{(-1)-curves}.
The question is of interest also for special configurations of points in $\P^2$
and we focus our research here on such configurations.
More concretely, we consider some examples of such special rational
surfaces and list all negative curves on them.
In particular, we study
blow ups of $\P^2$ at certain points which lie on elliptic curves.
Our main results classify negative curves on such surfaces; see Theorems
\ref{thm: very general on cubic}, \ref{thm: 3 torsion} and \ref{thm: Fermat}.
As a consequence, we show that Conjecture \ref{bnc}  holds for such surfaces.
This recovers some existing results of Harbourne and Miranda \cite{HarMir90}, \cite{Har97TAMS}.
Additionally we compute values of the number $b(X)$ on such surfaces.

\section{Very general points on an irreducible cubic}\label{sec:cubic}
   To put our results in Section \ref{sec: Fermat} into perspective,
   we recall results on negative curves on blow ups of $\P^2$ at
   $s$ very general points on an plane curve of degree $3$.
   Geometry of such surfaces was studied by Harbourne in \cite{Har85}.
\begin{theorem}[Points on a cubic curve]\label{thm: very general on cubic}
Let $D$ be an irreducible and reduced plane cubic and let
   $P_1,\ldots,P_s$
   be smooth points on $D$. Let $f: X \longrightarrow \proj^2$ be the blow up at
   $P_1,\ldots, P_s$. If $C \subset X$ is any reduced and irreducible curve such that
   $C^2 < 0$, then
\begin{itemize}
   \item[a)] $C$ is the proper transform of $D$, or
   \item[b)] $C$ is a $(-1)$-curve, or
   \item[c)] $C$ is a $(-2)$-curve.
\end{itemize}
   Moreover, if the points $P_1,\ldots,P_s$ are very general, then only cases a) and b) are possible.
\end{theorem}
\begin{proof}
   The first part of Theorem follows from \cite[Remark III.13]{Har97TAMS} and also from our Remark \ref{anti-canonical}.
   The "moreover" part follows from the following abstract argument.
    A negative curve on $X$ is either a component of $-K_X$, or a $(-1)$-curve or a $(-2)$-curve.
    But a $(-2)$-curve is in $\ker(\Pic(X)\to\Pic^0(-K_X))$, which is $0$ for very general points,
    so there are no $(-2)$-curves.
%
%
%
%
\end{proof}
\begin{corollary}
   Let $X$ be a surface as in Theorem \ref{thm: very general on cubic}
   with $s>0$ very general points. Then Conjecture \ref{bnc} holds for $X$ and we have
   $$b(X)=\min\left\{-1,\; 9-s \right\}.$$
\end{corollary}
\section{Special points on a smooth cubic}\label{sec: Fermat}
In this section, we consider blow ups of $\P^2$ at 3-torsion points of
an elliptic curve as well as the points of intersection of the Fermat
arrangement of lines.
   In order to consider these two cases, we deal first with the following numerical lemma which seems quite interesting
   in its own right.
\begin{lemma}\label{lem: nice}
 	Let $m_1,\dots,m_9$ be nonnegative real numbers satisfying the following 12 inequa\-li\-ties:
 	\begin{gather}\label{assumptions}
 	m_1+m_2+m_3 \leq 1,\\\label{1 1}
 	m_4+m_5+m_6 \leq 1,\\
 	m_7+m_8+m_9 \leq 1,\\
 	m_1+m_4+m_7 \leq 1,\\
 	m_2+m_5+m_8 \leq 1,\\
 	m_3+m_6+m_9 \leq 1,\\
 	m_1+m_5+m_9 \leq 1,\\
 	m_2+m_6+m_7 \leq 1,\\\label{1}
 	m_3+m_4+m_8 \leq 1,\\\label{2}
 	m_1+m_6+m_8 \leq 1,\\\label{3}
 	m_2+m_4+m_9 \leq 1,\\\label{4}
 	m_3+m_5+m_7 \leq 1.
 	\end{gather}
 	Then $m_1^2+ \dots + m_9^2 \leq 1$.
\end{lemma}
\begin{proof}
 	Assume that the biggest number among $m_1,\dots,m_9$ is $m_1=1-m$ for some $0\leq m\leq 1$.

Consider the following four pairs of numbers
   $$p_1=(m_2,m_3),\; p_2=(m_4,m_7),\; p_3=(m_9,m_5),\; p_4=(m_6,m_8).$$
   These are pairs such that together with $m_1$ they occur in one of the $12$ inequalities.
   In each pair one of the numbers is greater or equal than the other.
   Let us call this bigger number a \emph{giant}.
   A simple check shows that there are always three pairs, such that their giants
   are subject to one of the $12$ inequalities in the Lemma.

   Without loss of generality, let $p_1$, $p_2$, $p_3$ be such pairs.
   Also without loss of generality, let
   $m_2$, $m_4$ and $m_9$ be the giants.
   Thus $m_2+m_4+m_9 \leq 1$. Assume that also $m_6$ is a giant.

   Inequality $m_2+m_3 \leq m$ implies that
   $$m_2^2 + m_3^2 = (m_2+m_3)^2-2m_2m_3 \leq m(m_2+m_3)-2m_2m_3.$$
   Observe also that
   $$(m_2+m_3)^2-4m_2m_3 \leq m(m_2-m_3).$$
   Analogous inequalities hold for pairs $p_2, p_3$ and $p_4$.
   Therefore
\begin{gather*}
   m_2^2+m_3^2+m_4^2+m_7^2+m_5^2+m_9^2 \leq \\
   \leq m(m_2+m_4+m_9+m_3+m_7+m_5) - 2m_2m_3-2m_4m_7-2m_5m_9 \leq \\
   \leq m + \big[ m(m_3+m_7+m_5) - 2m_2m_3-2m_4m_7-2m_5m_9 \big].
\end{gather*}

   But we have also
\begin{gather*}
m_2^2+m_3^2+m_4^2+m_7^2+m_5^2+m_9^2 =\\
= (m_2+m_3)^2+(m_4+m_7)^2+(m_5+m_9)^2 - 2m_2m_3-2m_4m_7-2m_5m_9 = \\
= (m_2+m_3)^2-4m_2m_3+(m_4+m_7)^2-4m_4m_7+\\ +(m_5+m_9)^2-4m_5m_9 + 2m_2m_3+2m_4m_7+2m_5m_9 \leq \\
\leq m(m_2-m_3) + m(m_4-m_7) + m(m_9-m_5) + 2m_2m_3+2m_4m_7+2m_5m_9 \leq \\
\leq m - \big[ m(m_3+m_7+m_5) - 2m_2m_3-2m_4m_7-2m_5m_9 \big],
\end{gather*}
which obviously gives
$$m_2^2+m_3^2+m_4^2+m_7^2+m_5^2+m_9^2 \leq m.$$

   Since
   $$m_6^2+m_8^2 \leq m_6^2+m_6m_8 \leq m_6(m_6+m_8) \leq (1-m)m,$$
   we get that the sum of all nine squares is bounded by
\begin{equation*}
   (1-m)^2 + m + (1-m)m = 1. \qedhere
\end{equation*}
\end{proof}
   If we think of numbers $m_1,\ldots,m_9$ as arranged in a $3\times 3$ matrix
   $$\left(\begin{array}{ccccc}
      m_1 && m_2 && m_3\\
      m_4 && m_5 && m_6\\
      m_7 && m_8 && m_9
      \end{array}\right),$$
   then the inequalities in the Lemma \ref{lem: nice} are obtained
   considering the horizontal, vertical triples
   and the triples determined by the condition that there is exactly one element $m_i$ in every
   column and every row of the matrix (so determined by permutation matrices).
   Bounding sums of only such triples allows us to bound the sum of squares of all
   entries in the matrix. It is natural to wonder, if this phenomena extends to higher
   dimensional matrices. One possible extension is formulated as the next question.
\begin{problem}
   Let $M=\left(m_{ij}\right)_{i,j=1\ldots k}$ be a matrix whose
   entries are non-negative real numbers. Assume that all the horizontal, vertical
   and permutational $k$-tuples of entries in the matrix $M$ are bounded by $1$.
   Is it true then that the sum of squares of all entries of $M$ is also bounded by $1$?
\end{problem}
\subsection{Torsion points}\label{ssec: torsion points}
   We now consider a blow up of $\P^2$ at $9$ points which are torsion points of order $3$
   on an elliptic curve embedded as a smooth cubic.
\begin{theorem}[$3$--torsion points on an elliptic curve]\label{thm: 3 torsion}
   Let $D$ be a smooth plane cubic and let $P_1,\ldots,P_9$ be the flexes of $D$.
   Let $f:X\to\P^2$ be the blow up of $\P^2$ at $P_1,\ldots,P_9$.
   If $C$ is a negative curve on $X$, then
   \begin{itemize}
      \item[a)] $C$ is the proper transform of a line passing through two (hence three)
      of the points $P_1,\ldots,P_9,$ and $C^2=-2$ or
      \item[b)] $C$ is an exceptional divisor of $f$ and $C^2=-1$.
   \end{itemize}
\end{theorem}
\proof
   It is well known that there is a group law on $D$ such that the flexes are $3$--torsion points.
   Since any line passing through
   two of the torsion points automatically meets $D$ in a third torsion point, there are altogether
   $12$ such lines. The torsion points form a subgroup of $D$ which is isomorphic to $\Z_3\times \Z_3$.
   We can pick this isomorphism so that
   $$P_1=(0,0),\; P_2=(1,0),\; P_3=(2,0),$$
   $$P_4=(0,1),\; P_5=(1,1),\; P_6=(2,1),$$
   $$P_7=(0,2),\; P_8=(1,2),\; P_9=(2,2).$$
   This implies that the following triples of points are collinear (note that these are exactly triples of indices
   in inequalities from \eqref{1 1} to \eqref{2}:
   $$(P_1,P_2,P_3),\; (P_4,P_5,P_6),\; (P_7,P_8,P_9),\; (P_1, P_4, P_7),$$
   $$(P_2,P_5,P_8),\; (P_3,P_6,P_9),\; (P_1,P_5,P_9),\; (P_2, P_6, P_7),$$
   $$(P_3,P_4,P_8),\; (P_1,P_6,P_8),\; (P_2,P_4,P_9),\; (P_3, P_5, P_7).$$
   Let $C$ be a reduced and irreducible curve on $X$ different from the exceptional divisors of $f$
   and the proper transforms of lines through the torsion points. Then $C$ is of the form
   $$C=dH-k_1E_1-\ldots-k_9E_9,$$
   where $E_1,\ldots,E_9$ are the exceptional divisors of $f$ and
   $k_1,\ldots,k_9 \ge 0$ and $d> 0$
   is the degree of the image $f(C)$ in $\P^2$.

   For $i=1,\ldots, 9$, let $m_i=\frac{k_i}{d}$. Since
   $C$ is different from proper transforms of the $12$ lines distinguished above,
   taking the intersection product of $C$ with the 12 lines, and dividing by $d$, we obtain exactly the $12$ inequalities in Lemma \ref{lem: nice}.
   The conclusion of Lemma \ref{lem: nice} implies then that
   $$C^2=d^2-\sum_{i=1}^9m_i^2\geq 0,$$
   which finishes our argument.
\endproof
\begin{corollary}\label{cor:bnc on 3-torsion}
   For the surface $X$ in Theorem \ref{thm: 3 torsion} Conjecture \ref{bnc} holds with
   $$b(X)=-2.$$
\end{corollary}
\begin{remark}\label{rem: fibrations}
   Theorem \ref{thm: 3 torsion} fits in a more general setting of elliptic fibrations. Negative curves
   on surfaces $X$ with $h^0(X, -mK_X)\geq 2$ for some $m\geq 2$ have been studied by Harbourne and Miranda in \cite{HarMir90}.
\end{remark}
   The observation in Remark \ref{rem: fibrations} allows us to explain results of Theorem \ref{thm: 3 torsion}
   from another point of view. Let $\wtilde{D}$ be the proper transform of $D$. Then, it is a member
   of the Hesse pencil, see \cite{ArtDol09}, in particular the linear system $|\wtilde{D}|$ defines
   a morphism from $X$ to $\P^1$. The components of reducible fibers are $(-2)$ curves. There are 12 of them and they are proper transforms of lines passing through triples of blown-up points.
   The exceptional divisors over these points are the $(-1)$ curves. These are sections of the fibration determined by $\wtilde{D}$.

   Clearly Corollary \ref{cor:bnc on 3-torsion} follows also from the adjunction and the fact that $-K_X$ is effective, see Remark \ref{anti-canonical}.
   Of course, there is no reason to restrict to $3$--torsion points.
\begin{remark}
   With the same approach one can show that $m\geq 4$ the Bounded Negativity Conjecture holds on the blow ups of $\P^2$
   at all the $m$--torsion points of an elliptic curve
   embedded as a smooth cubic and we have
   $$b(X)=9-m^2.$$
\end{remark}
\subsection{Fermat configuration of points}
   The $9$ points and $12$ lines considered in subsection \ref{ssec: torsion points} form the famous
   Hesse arrangement of lines;  see \cite{Hir83}.
   Any such arrangement is projectively equivalent to that obtained from the flex points of the Fermat cubic
   $x^3+y^3+z^3=0$ and the lines determined by their pairs. Explicitly in coordinates we have then
   $$P_1=(1:\eps:0),\; P_2=(1:\eps^2:0),\; P_3=(1:1:0),$$
   $$P_4=(1:0:\eps),\; P_5=(1:0:\eps^2),\; P_6=(1:0:1),$$
   $$P_7=(0:1:\eps),\; P_8=(0:1:\eps^2),\; P_9=(0:1:1),$$
   for the points and
   $$x=0,\; y=0,\; z=0,\; x+y+z=0, x+y+\eps z=0,\; x+y+\eps^2 z=0$$
   $$x+\eps y+z=0,\; x+\eps^2 y+z=0,\; x+\eps y+\eps z=0,\; x+\eps y+\eps^2 z=0,\; x+\eps^2 y+\eps z=0, x+\eps^2 y+\eps^2 z=0,$$
   for the lines, where $\eps$ is a primitive root of unity of order $3$.

   Passing to the dual plane, we obtain an arrangement of $9$ lines
   defined by the linear factors of the Fermat polynomial
   $$(x^3-y^3)(y^3-z^3)(z^3-x^3)=0.$$
   These lines intersect in triples in $12$ points, which are dual to the lines of the Hesse arrangement.
   The resulting dual Hesse configuration has the type $(9_4, 12_3)$ and it belongs to a much
   bigger family of Fermat arrangements; see \cite{Szp19c}. Figure \ref{fig: dual Hesse}
   is an attempt to visualize this arrangement (which cannot be drawn in the real plane
   due to the famous Sylvester-Gallai Theorem; for instance, see \cite{Mel41}).
\begin{figure}[h]
	\begin{center}
		\includegraphics[width=0.50\textwidth]{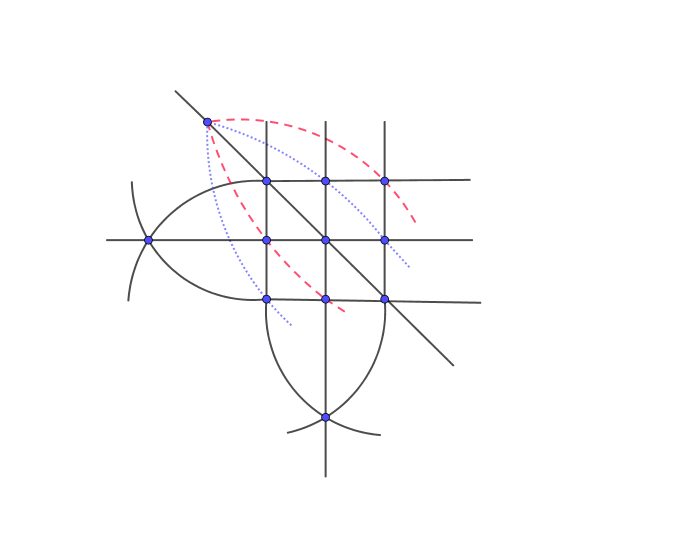}
	\end{center}
\vskip-35pt
\caption{Fermat configuration of points}
\label{fig: dual Hesse}
\end{figure}

   It is convenient to order the $9$ intersection points in the affine part in the following way:
$$
\begin{array}{lll}
   Q_1=(\eps: \eps: 1),& Q_2=(1: \eps: 1),& Q_3=(\eps^2: \eps: 1),\\
   Q_4=(\eps: 1: 1),& Q_5=(1: 1: 1),& Q_6=(\eps^2: 1: 1),\\
   Q_7=(\eps: \eps^2: 1),& Q_8=(1: \eps^2: 1),& Q_9=(\eps^2: \eps^2: 1).
\end{array}
$$
   With this notation established, we have the following result.
\begin{theorem}[Fermat points]\label{thm: Fermat}
   Let $f:X\to\P^2$ be the blow up of $\P^2$ at $Q_1,\ldots,Q_9$.
   If $C$ is a negative curve on $X$, then
   \begin{itemize}
      \item[a)] $C$ is the proper transform of a line passing through two or three
      of the points $Q_1,\ldots,Q_9,$ or
      \item[b)] $C$ is a $(-1)$-curve.
   \end{itemize}
\end{theorem}
\proof
   The proof of Theorem \ref{thm: 3 torsion} works with very few adjustments.

   Let us assume, to begin with, that $C$ is a negative curve on $X$,
   distinct from the curves
   listed in the theorem. Then
   $$C=dH-k_1E_1-\ldots-k_9E_9,$$
   for some $d>0$ and $k_1,\ldots,k_9\geq 0$. We can also assume that $d$ is the smallest
   number for which such a negative curve exists.
   As before, we set
   $$m_i=\frac{k_i}{d}\;\mbox{ for }\; i=1,\ldots,9.$$
   Then the inequalities \eqref{assumptions} to
   \eqref{1} follow from the fact that $C$ intersects the $9$ lines in
   the arrangement non-negatively.

   If one of the remaining inequalities \eqref{2}, \eqref{3} or \eqref{4} fails,
   then we perform a standard Cremona transformation based on the points involved
   in the failing inequality. For example, if \eqref{2} fails, we make Cremona based
   on points $Q_1, Q_6$ and $Q_8$. Note that these points are \emph{not} collinear
   in the set-up of our Theorem. Since $C$ is assumed not to be a line through any
   two of these points, its image $C'$ under Cremona is a curve of strictly lower degree,
   negative on the blow up of $\P^2$ at the $9$ points. The points $Q_1,\ldots,Q_9$
   remain unchanged by the Cremona because, as already remarked, all dual Hesse
   arrangements are projectively equivalent. Then $C'$ is again a negative curve
   on $X$ of degree strictly lower than $d$, which contradicts our
   choice of $C$ such that $C \cdot H$ is minimal.

   Hence, we can assume that the inequalities \eqref{2}, \eqref{3} and \eqref{4}
   are also satisfied. Then we conclude exactly as in the proof of Theorem \ref{thm: 3 torsion}.
\endproof

\begin{remark}\label{non-extremal} The surface $X$ considered in Theorem \ref{thm: Fermat}
  is a \textit{non-extremal Jacobian rational elliptic surface} and
  contains infinitely many $(-1)$-curves. See \cite{HarMir90} for more
  details.
\end{remark}

   In fact, we are in the position to identify all these $(-1)$-curves.
   Let $L(X,Y)$ denote the line determine two distinct points $X$ and $Y$.
   Let
   $$\call=\left\{U_1=L(Q_1,Q_6), U_2=L(Q_1,Q_8), U_3=L(Q_6,Q_8),
   V_1=L(Q_2,Q_4),\right.$$$$\left. V_2=L(Q_2,Q_9), V_3=L(Q_4,Q_9),
   W_1=L(Q_3,Q_5), W_2=L(Q_3,Q_7), W_3=L(Q_5,Q_7)\right\}$$
   be the set of lines determined by pairs of points $Q_i, Q_j$ with
   $1\leq i<j\leq 9$ which contain only $2$ points $Q_k$.
   These lines can grouped in three ''triangles'', which is indicated
   by the letters $U, V$ and $W$ used to labeling relevant triples.
   Vertices of these triangles determine standard Cremona transformations,
   which we denote by $\phi_1$ for the $U$-triangle, i.e., points
   $Q_1, Q_6, Q_8$ and $\phi_2$ and $\phi_3$ for the $V$ and $W$-triangles respectively.
\begin{corollary}\label{prop: all -1 curves}
   Let $C\subset X$ be a $(-1)$-curve. Then either $C\in\call$ or there exists a positive integer $r\geq 1$ and
   a sequence of Cremona transformations $\phi=\phi_{i_r}\circ\ldots\circ\phi_{i_1}$ with $i_1,\ldots,i_r\in\left\{1,2,3\right\}$
   such that $C$ is the image under $\phi$ of one of the lines in $\call$.
\end{corollary}
\proof
   The statement follows directly from the proof of Theorem \ref{thm: Fermat}.
   There is an interesting regularity in applying Cremona transformation, which we would like to present additionally.
   This is done best by the way of an example. Recall that there is the following general rule concerning
   changes of degree and multiplicities, when applying Cremona transformation. Let $\phi$ be the standard Cremona
   transformation based on a triangle $F, G, H$. Let $C$ be a curve of degree $d$, different from the three lines
   $L(F,G)$, $L(F,H)$ and $L(G,H)$ passing through the points $F, G, H$ with multiplicities $m_F, m_G, m_H$.
   Let $k=d-m_F-m_G-m_H$.
   Then the image curve $C'=\phi(C)$ has degree $d+k$ and multiplicities $m_F+k$, $m_G+k$, $m_H+k$ in the base
   points of the reverse Cremona transformation.
   In Table \ref{tab: Cremona} we present how the line $L(Q_1,Q_6)$ transforms under the sequence of
   Cremona transformations indicated in the first column. If it is possible to perform one of $2$ Cremonas, we indicate
   it by writing the chosen one in boldface. Of course, it is always possible to choose the Cremona
   performed in the last step but as this leads to nothing new, we ignore this option.
   \renewcommand*{\arraystretch}{1.5}
   \begin{table}\label{tab: Cremona}
   $$
   \begin{array}{|c||c||ccc||ccc||ccc|}
   \hline
   Cremona & \deg & Q_1 & Q_6 & Q_8 & Q_2 & Q_4 & Q_9 & Q_3 & Q_5 & Q_7\\
   \hline
   \hline
           & 1    & 1   & 1   & 0   & 0   & 0   & 0   & 0   & 0   & 0\\
   \hline
   \boldsymbol{\phi_2} & 2    & 1   & 1   & 0   & 1   & 1   & 1   & 0   & 0   & 0\\
   \hline
               \phi_3  & 4    & 1   & 1   & 0   & 1   & 1   & 1   & 2   & 2   & 2\\
   \hline
   \boldsymbol{\phi_1} & 6    & 3   & 3   & 2   & 1   & 1   & 1   & 2   & 2   & 2\\
   \hline
               \phi_2  & 9    & 3   & 3   & 2   & 4   & 4   & 4   & 2   & 2   & 2\\
   \hline
   \boldsymbol{\phi_3} & 12   & 3   & 3   & 2   & 4   & 4   & 4   & 5   & 5   & 5\\
   \hline
   \end{array}$$
   \caption{A series of Cremona transformations}
   \end{table}
   The diagram in Figure \ref{fig: diagram} indicates possible bifurcations at the places where one of two Cremona transformations
   can be performed. For simplicity, we put only degree of resulting $(-1)$-curves in the diagram.
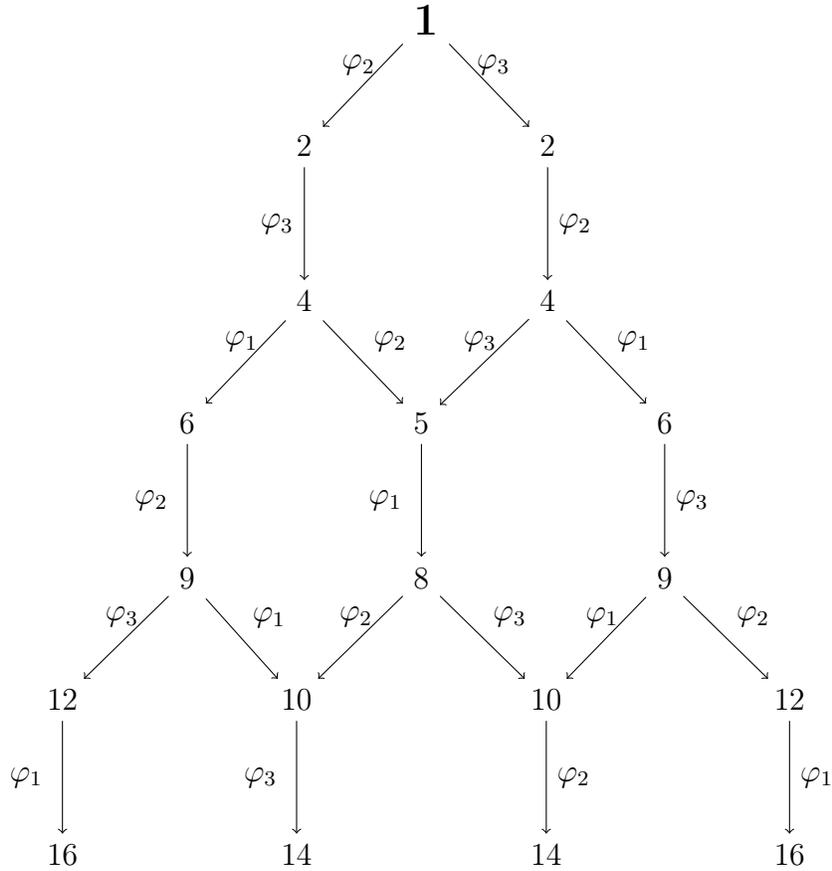
\begin{figure}
\begin{tikzpicture}[->,auto,node distance=1.5cm,main node/.style={font=\Large\bfseries}]
  %
  \node[main node] (C) {1};
  \node (CL) [below left = of C] {2};
  \node (CR) [below right = of C] {2};
  \node (CLC) [below = of CL] {4};
  \node (CRC) [below = of CR] {4};
  \node (CLCL) [below left = of CLC] {6};
  \node (CRCR) [below right = of CRC] {6};
  \node (CC) [below right = of CLC] {5};
  \node (CLCLC) [below = of CLCL] {9};
  \node (CCC) [below = of CC] {8};
  \node (CRCRC) [below = of CRCR] {9};
  \node (CRCRCR) [below right = of CRCRC] {12};
  \node (CCCL) [below left = of CCC] {10};
  \node (CCCR) [below right = of CCC] {10};
  \node (CLCLCL) [below left  = of CLCLC] {12};
  \node (CLCLCLC) [below = of CLCLCL] {16};
  \node (CCCLC) [below = of CCCL] {14};
  \node (CCCRC) [below = of CCCR] {14};
  \node (CRCRCRC) [below = of CRCRCR] {16};
  \path (C) edge node[above] {$\phi_2\;$} (CL)
        (C) edge node[above] {$\;\phi_3$} (CR)
        (CL) edge node[left] {$\phi_3$} (CLC)
        (CR) edge node {$\phi_2$} (CRC)
        (CLC) edge node[above] {$\phi_1\;$} (CLCL)
        (CLC) edge node {$\phi_2$} (CC)
        (CRC) edge node[above] {$\phi_3\;$} (CC)
        (CRC) edge node {$\phi_1$} (CRCR)
        (CLCL) edge node[left] {$\phi_2\;$} (CLCLC)
        (CC) edge node[left] {$\phi_1\;$} (CCC)
        (CRCR) edge node {$\phi_3\;$} (CRCRC)
        (CLCLC) edge node[above] {$\phi_3\;$} (CLCLCL)
        (CLCLC) edge node {$\phi_1\;$} (CCCL)
        (CCC) edge node[above] {$\phi_2\;$} (CCCL)
        (CCC) edge node {$\phi_3\;$} (CCCR)
        (CRCRC) edge node[above] {$\phi_1\;$} (CCCR)
        (CRCRC) edge node {$\phi_2\;$} (CRCRCR)
        (CLCLCL) edge node[left] {$\phi_1\;$} (CLCLCLC)
        (CCCL) edge node[left] {$\phi_3\;$} (CCCLC)
        (CCCR) edge node {$\phi_2\;$} (CCCRC)
        (CRCRCR) edge node {$\phi_1\;$} (CRCRCRC);
\end{tikzpicture}
\caption{Bifurcations of Cremona transformations}
\label{fig: diagram}
\end{figure}
\endproof
   The diagram in Figure \ref{fig: diagram} seems quite interesting in its own right.
   There is a vertical symmetry, and it leads to a scheme of numbers indicated in Table \ref{tab: diagram},
   which has some reminiscences to the Pascal's triangle.
   \renewcommand*{\arraystretch}{1.5}
   \begin{table}
   $$
   \begin{array}{ccccccc}
    &&& 1 &&&\\
    && 2 && 2 &&\\
    && 4 && 4 &&\\
    & 6 && 5 && 6 &\\
    & 9 && 8 && 9 &\\
    12 && 10 && 10 && 12\\
    16 && 14 && 14 && 16
   \end{array}
   $$
   \caption{Cremona hexal}
   \label{tab: diagram}
   \end{table}
\begin{problem}
   Investigate numerical properties of the Cremona hexal. For example, find a direct formula
   for the entry in line $i$ and column $j$.
\end{problem}
\begin{remark}\label{anti-canonical}
   If we are interested only in the bounded negativity
   property on $X$, then there is a simple proof. Indeed,
if $C \subset X$ is a reduced and irreducible curve, the genus formula
gives
$$1+\frac{C\cdot (C+K_X)}{2}\geq 0.$$
Now, since the anti-canonical divisor on the blow up of $\P^2$
   in the $9$ Fermat points is effective, we conclude that $C$ is a
   component of $-K_X$ or
   $$C^2\geq -2-CK_X\geq -2.$$
\end{remark}
Having classified all the negative curves on the blow up of $\P^2$ at the
9 Fermat points, it is natural to wonder about the negative curves on
blow ups of $\P^2$ arising from the other Fermat configurations.
Note that the argument given in Remark \ref{anti-canonical} is no
longer valid, since $-K_X$ is not nef or effective anymore. So it will
be interesting to ask whether BNC holds for such surfaces. We pose the following problem.

\begin{problem}\label{pro: Fermat}
   For a positive integer $m$, let $Z(m)$ be the set of all points of the form
   $$(1:\eps^{\alpha}:\eps^{\beta}),$$
   where $\eps$ is a primitive root of unity of order $m$ and $1 \le
   \alpha,\beta \le m$.
   Let $f_m:X(m)\to\P^2$ be the blow up of $\P^2$ at all the points of $Z(m)$.
   Is the negativity bounded on $X(m)$? If so, what is the value of $b(X(m))$?
\end{problem}

We end this note by the following remark which discusses bounded
negativity for blow ups of $\proj^2$ at 10 points.
\begin{remark}
Let $X$ denote a blow up of $\proj^2$ at 10 points. As mentioned
before, if the blown up points are general, then Conjecture
\ref{(-1)-curves} predicts that the only negative curves on $X$ are
$(-1)$-curves. This is an open question. On the other hand, let us consider a couple of examples
of special points.

Let $X$ be obtained by blowing up the 10 nodes of an irreducible and
reduced rational nodal sextic. Such surfaces are called \textit{Coble
  surfaces} (these are smooth rational surfaces $X$ such that
$|-K_X| = \emptyset$, but $|-2K_X| \ne \emptyset$). Then it is known
that BNC holds for $X$. In fact, we have $C^2 \ge -4$ for every
irreducible and reduced curve $C\subset X$; see \cite[Section
3.2]{CD}.

Now let $X$ be the blow up of 10 double points of intersection of 5
general lines in $\proj^2$. Then $-K_X$ is a big divisor and by \cite[Theorem
1]{TVV}, $X$ is a \textit{Mori dream space}. For such surfaces, the
submonoid of the Picard group generated by the effective classes
is finitely generated. Hence BNC holds for $X$
(\cite[Proposition I.2.5]{Har10}).
\end{remark}

\textbf{Acknowledgements:}
A part of this work was done when KH visited the
Pedagogical University of Krakow in October 2018. He is grateful to
the university and the department of mathematics for making it a wonderful visit.
This research stay of KH was partially supported by the Simons Foundation
and by the Mathematisches Forschungsinstitut Oberwolfach and he is
grateful to them. The authors thank the referee for making several
useful comments which improved this note. Finally, the authors warmly
thank Brian Harbourne for suggesting corrections and ameliorations
which substantially improved the paper.



\end{document}